\documentclass[12pt, oneside]{amsart}

\usepackage{amssymb, amsmath, amsthm,geometry}
\usepackage[english]{babel}
\usepackage{mathrsfs} 

\usepackage{geometry}

\usepackage{bm}
\usepackage{subfig,tikz}
\usepackage{wrapfig}
\usepackage{color}
\usepackage{xspace}
\usetikzlibrary{shapes.geometric}

\usepackage{graphicx}
\usepackage[colorlinks=true, citecolor=blue, linkcolor=blue, urlcolor=red]{hyperref}
 \usepackage{xcolor}
\usepackage{hyperref}

\usepackage{lineno}

\newtheorem{thm}{Theorem}[section]

\newtheorem{lem}[thm]{Lemma}

\newtheorem{prop}[thm]{Proposition}

\numberwithin{equation}{section}

\newcommand{\bel}{\begin{equation} \label}
\newcommand{\ee}{\end{equation}}
\def\beq{\begin{equation}}
\def\eeq{\end{equation}}
\newcommand{\bea}{\begin{eqnarray}}
\newcommand{\eea}{\end{eqnarray}}
\newcommand{\beas}{\begin{eqnarray*}}
\newcommand{\eeas}{\end{eqnarray*}}
\newcommand{\pd}{\partial}

\newcommand{\re}{\mathfrak R}

\newcommand{\R}{\mathbb{R}}
\newcommand{\A}{\mathcal{A}}

\newcommand{\C}{\mathbb{C}} 
\newcommand{\D}{\mathcal D}
\newcommand{\E}{\mathcal E}
\newcommand{\CI}{\mathcal{C}}
\newcommand{\N}{\mathbb{N}}

\newcommand{\M}{\mathcal{M}}
\newcommand{\g}{\bar{g}}

\renewcommand{\div}{\mathrm{div}\,}  

\def\inter{\text{int}}
\newcommand{\pair}[1]{\langle #1 \rangle}

\newcommand{\supp}{\mathrm{supp}\,}  

\allowdisplaybreaks

\def\epsilon{\varepsilon}
\def\phi {\varphi}

\newcommand{\abs}[1]{\lvert#1\rvert}

\renewcommand{\leq}{\leqslant}
\renewcommand{\geq}{\geqslant}
\providecommand{\abs}[1]{\left\lvert#1\right\rvert}
\providecommand{\norm}[1]{\left\lVert#1\right\rVert}

\DeclareMathOperator{\dis}{dist}

\numberwithin{equation}{section}

\title[Dirichlet to Neumann map]{Recovery of non-smooth coefficients appearing in anisotropic wave equations}

\author[]{Ali Feizmohammadi}
\address{Department of Mathematics, University College London, London, UK-WC1E  6BT, United Kingdom}
\email{a.feizmohammadi@ucl.ac.uk}
\author[]{Yavar Kian}
\address{Aix Marseille Universit\'{e}, Universit\'{e} de Toulon, CNRS, CPT, Marseille, France}
\email{yavar.kian@univ-amu.fr}

\begin{document}
\maketitle

\begin{abstract}
We study the problem of unique recovery of a non-smooth one-form $\A$ and a scalar function $q$ from the Dirichlet to Neumann map, $\Lambda_{\A,q}$, of a hyperbolic equation on a Riemannian manifold $(M,g)$. We prove uniqueness of the one-form $\A$ up to the natural gauge, under weak regularity conditions on $\A,q$ and under the assumption that $(M,g)$ is simple. Under an additional regularity assumption, we also derive uniqueness of the scalar function $q$. The proof is based on the geometric optic construction and inversion of the light ray transform extended as a Fourier Integral Operator to non-smooth parameters and functions.
\end{abstract} 
\tableofcontents
\section{Introduction}
Let $T>0$, and let $(M,g)$ denote a compact connected smooth $n$-dimensional Riemannian manifold with smooth boundary $\pd M$. We consider the Lorentzian manifold $(\M,\g)$ defined as $\M=(0,T)\times M$ with the metric $\g=-(dt)^2+g$. Let $\div_{\g}$ (resp., $\nabla^{\g}$) denote the divergence operator (resp., gradient operator) on $(\M,\g)$ and define the Laplace-Beltrami operator associated to $(\M,\g)$ through $\Delta_{\g}\cdot=\div_{\g}\nabla^{\g}\cdot$. In local coordinates $(t=x^0,x^1,\ldots,x^n)=(t,x)$, we have
$$\Delta_{\g}\cdot=\sum_{i,j=0}^n \frac{1}{\sqrt{|\g|}}\pd_i(\sqrt{|\g|}{\g}^{ij}\pd_j\cdot)=(-\pd^2_t+\Delta_{g})\cdot,$$
where $\Delta_g$ is analogously defined on $(M,g)$. In this paper, we will make the standing assumption that $(M,g)$ is {\em simple}, that is to say, it is simply connected, any geodesic in $M$ has no conjugate points and the boundary $\pd M$ is strictly convex in the sense that the second fundamental form is positive for every point on the boundary. Any two points in a simple manifold can be connected through a unique geodesic.

We consider a scalar function $q$ and a one-form $\A$ on $(\M,\g)$. In local coordinates, we have
\bel{Alocal}\A(t,x)=b(t,x) \,dt+\sum_{i=1}^na_j(t,x)\,dx^j= b(t,x)\,dt+A(t,x),\ee
where $A$ is a time-dependent one-form on $(M,g)$. Throughout this paper we impose the following regularity assumptions on these coefficients:
\bel{regularity}
\begin{aligned}
&\A \in W^{1,1}(0,T;L^{2}(M;T^*\M))\, \cap\, \CI(\M;T^*\M)\\
&\div_{\g}\, \A \in L^{p_1}(0,T;L^{p_2}(M)),\\
& q \in L^{p_1}(0,T;L^{p_2}(M)).
\end{aligned}
\ee
where $p_1>1$ and $p_2 \in [n,\infty]\setminus \{2\}$. We consider the initial boundary value problem (IBVP)
\bel{eq1}
\left\{ \begin{array}{ll} L_{\A,q}u:= -\Delta_{\g} u + \A\nabla^{\g} u+ q u  =  0, & \mbox{on}\ \M,\\  
u  =  f, & \mbox{on}\ (0,T)\times \pd M,\\  
 u(0,\cdot)  =0,\quad  \pd_tu(0,\cdot)  =0 & \mbox{on}\ M,
\end{array} \right.
\ee 
This problem is well-posed for any $f \in H^1_0((0,T]\times \pd M)$ (see Section~\ref{Forwardproblem}) and admits a unique solution $u$ in the energy space
\bel{energyspace}
\mathscr{X}:=\CI^1(0,T;L^2(M)) \cap \CI(0,T;H^1(M)). 
 \ee 
We define the Dirichlet-to-Neumann (DN in short) map 
\bel{dnmapspace}
\Lambda_{\A,q}:H^{1}_0((0,T]\times \pd M)\ni f\mapsto \left(\partial_{\bar\nu} u-\frac{\A\bar{\nu}}{2}u\right)\,|_{(0,T)\times \pd M}\in L^2((0,T)\times \pd M)
\ee 
for equation \eqref{eq1}. Here $\bar{\nu}$ represents the outward normal unit vector to $(0,T)\times \pd M$. We refer the reader to Sections~\ref{Forwardproblem}-\ref{DNformulation} for a rigorous presentation of the direct problem \eqref{eq1} and this formulation of the DN map. In this paper, we are interested in determining the unknown complex valued coefficients $\A,q$, given the map $\Lambda_{\A,q}$, up to the natural obstructions for this problem as discussed in \cite[Section 1.2]{FIKO}. 

\subsection{Main results}
\label{main}
Before stating the main theorem, we need to define the set $\mathcal E\subset \M$ where we recover the coefficients. Let us define the domain of influence 
$$\mathcal D:=\{(t,x)\in \M\,|\, \dis{(x,\pd M)}<t<T-\dis{(x,\pd M)}\}.$$
By finite speed of propagation, no information can be obtained about the coefficients $\A,q$ from $\Lambda_{\A,q}$ on the set $\M\setminus \mathcal D$. Thus, $\mathcal D$ represents the maximal set where one can, in theory, recover the coefficients. Now, for $T>2\,\textrm{Diam}(M)$, we start by fixing a subset of $\mathcal D$ given by
$$ \mathcal E:=\{(t,x)\in \M\,|\, D_g(x)<t<T-D_g(x)\},$$
where $D_g(x)$ denotes the length of the longest geodesic passing through the point $x$ in $M$. Since $(M,g)$ is simple, this is a well-defined continuous function on $M$. With the definition of $\E$ complete, we can state the main results in our paper as follows:
\bigskip
\begin{thm}\label{t1} Suppose $T>2\,\textrm{Diam}(M)$ and that $(M,g)$ is a simple Riemannian manifold. Let $\A_1,\A_2$ denote one-forms and $q_1, q_2$ denote scalar functions satisfying \eqref{regularity} and such that
\bel{support}\supp{(\A_1-\A_2)} \subset \mathcal E\quad\text{and}\quad \supp{(q_1-q_2)} \subset \mathcal E.\ee 
Then the condition
\bel{t1b}\Lambda_{\A_1,q_1}=\Lambda_{\A_2,q_2}\ee
implies that there exists $\psi\in \mathcal C^1(\M)$ with $\psi|_{\pd \M}=0$ such that 
\bel{t1c}\A_1=\A_2+\,\bar{d}\psi \quad \forall \,(t,x)\in \mathcal M,\ee
where $\bar d$ denotes the exterior derivative on $\M$.
\end{thm}
\bigskip
\begin{thm}\label{t2} Let the hypothesis of Theorem~\ref{t1} be fulfilled and assume additionally that 
\bel{additional}
q_1-q_2\in L^{p_1}(0,T;L^{\infty}(M)),\quad \div_{\g}(\A_1-\A_2) \in L^{p_1}(0,T;L^{\infty}(M))
\ee
 holds. Then the condition $\Lambda_{\A_1,q_1}=\Lambda_{\A_2,q_2}$ implies that there exists $\psi \in \CI_0^1(\M)$ with $\Delta_{\g}\psi \in L^{p_1}(0,T;L^{\infty}(M))$ such that
\bel{t2c}
\A_1=\A_2+\,\bar{d}\psi, \quad q_1=q_2+\frac{1}{2}\Delta_{\g}\psi-\frac{1}{2}\A_2\nabla^{\g}\psi-\frac{1}{4}\langle\nabla^{\g}\psi,\nabla^{\g}\psi\rangle_{\g}\quad \forall\,(t,x) \in \M.
\ee
\end{thm}
\bigskip
The proof of Theorems~\ref{t1}-\ref{t2} rely in part on the inversion of the light ray transform of one-forms and scalar functions over $\M$ under the hypothesis  \eqref{support} and the regularity conditions \eqref{regularity}. This has already been accomplished for $\CI^1$ one-forms and continuous scalar functions in \cite{FIKO}, but some additional analysis is required here as we are working with a wider regularity class for the coefficients $\A$ and $q$. Let us briefly recall the notion of the light ray transform here. We denote by $SM \subset TM$ the unit sphere bundle of $(M,g)$, and by $\gamma(\cdot; x, v)$ the geodesic with the initial data $(x,v) \in SM$. For all $(x,v) \in SM^\inter$, we define the exit times
$$
\tau_{\pm}(x,v) = \inf \{ r > 0 : \gamma(\pm r; x,v) \in \pd M\}
$$
and note that since $(M,g)$ is simple, we have $\tau_{\pm}(x,v) < \textrm{Diam}(M)$. Define
$$
\pd_{\pm} SM = \{ (x,v) \in SM \,|\, x \in \pd M\quad \pm \pair{v, \nu(y)}_g > 0\}.
$$
All geodesics in $M^\inter$ can be parametrized by $\gamma(\cdot; x,v)$, $(x,v) \in \pd_- SM$.
The geodesic ray transform on $(M,g)$ is defined
for $f \in \CI^\infty(M)$ by
$$
\mathcal I f(x,v) = \int_0^{\tau_+(x,v)} f(\gamma(r;x,v)) dr, \quad (x,v) \in \pd_- SM.
$$
Next, we consider the Lorentzian manifold $\R\times M$ with metric $\g=-(dt)^2+g$. Recall that a curve $\beta$ in $\R \times M$ is called a null geodesic (also called light rays) if 
\bel{null geo}
\nabla^{\g}_{\dot{\beta}}\dot{\beta}=0 \quad \text{and}\quad \langle \dot{\beta},\dot{\beta}\rangle_{\g}=0.
\ee
One can use the product structure of the Lorentzian manifold $\R\times M$ to see that the null geodesics $\beta$ can be parametrized as 
$$\beta(r;s,x,v)=(r+s,\gamma(r;x,v)) \quad \forall (s,x,v) \in \R\times \pd_-SM.$$
Thus, we can identify null geodesics $\beta$ through $\beta(\cdot;s,x,v)$ with $(s,x,v) \in \R \times \pd_-SM$ over their maximal intervals $[0,\tau_+(x,v)]$. We define the light ray transform on $\R \times M$ that is defined for $f \in \CI^\infty(\R \times M)$ by 
$$
\mathcal L f(s,x,v) = \int_0^{\tau_+(x,v)} f(r+s,\gamma(r;x,v)) \, dr, \quad \forall (s,x,v) \in \R \times \pd_- SM. 
$$
Similarly, we define the light ray transform corresponding to smooth one-forms $\mathcal B$, through the expression
$$
\mathcal L \mathcal B\,(s,y,v) := \mathcal L \mathcal (\mathcal B\dot{\beta})\, (s,y,v).
$$
We will sometimes use the short hand notation $\mathcal L_{\beta} f,\mathcal L_{\beta}\mathcal B$ in place of the above notation. In Section~\ref{FIO}, we will show that $\mathcal L_{\beta} f$ is a Fourier Integral Operator and that the domain of definition can be extended to $L^{p}$ spaces. We will prove the following Proposition in Section~\ref{lightray section}, that is a key step in proving Theorems~\ref{t1}-\ref{t2}.
\begin{prop}
\label{lightray}
Let $f \in L^1(0,T;L^2(M))$ and $\mathcal B \in \CI(\M;T^*\M)$ both vanish on the set $\M \setminus \mathcal E$. Then the following statements hold:
\begin{itemize}
\item[(i)] {If $\mathcal L_{\beta}\,f=0$ for all maximal null geodesics $\beta \subset \mathcal D$, then $f\equiv 0$.}
\item[(ii)]{If $\mathcal L_{\beta}\,\mathcal B=0$ for all maximal null geodesics $\beta \subset \mathcal D$, then $\mathcal B\equiv \bar{d}\psi$ for some $\psi \in \CI^1(\M)$ with $\psi|_{\pd \M}=0$.}
\end{itemize}
\end{prop}

The proof of statement (ii) will be identical to that of statement (ii) in \cite[Proposition 1.4]{FIKO}, with the only difference being that $\mathcal B \in \CI(\M;T^*\M)$ here as opposed to $\CI^1(\M;T^*\M)$. Reproducing the exact same analysis as in the proof there shows that one obtains existence of a $\psi \in \CI^1(\M)$ with $\psi |_{\pd \M}=0$ such that (ii) holds and therefore for the sake of brevity we omit this proof. We will however prove statement (i) in Section~\ref{lightray section}. 

\subsection{Previous literature}

Historically, uniqueness results for the recovery of coefficients can be divided into two categories, based on whether or not the geometry and coefficients are dependent on time. The time-independent case has been studied extensively and one can outline at least three general methods for the recovery of the coefficients in this case. The first approach, stemming from the seminal works \cite{Bel87,Bel92}, relies on the so-called Boundary Control (BC) method together with Tataru's sharp unique continuation theorem \cite{Tataru}. This method yields recovery of time-independent coefficients under very weak assumptions on the transversal manifold $(M,g)$. We refer the reader to \cite{KOM} for an introduction to the BC method and to the recent paper \cite{KOP} for an example of a state of the art result and finally to \cite{Bel,KKL} for review. The stability results are in general double logarithmic (\cite{BKL}) although in \cite{LOk} a stronger low-frequency stability estimate was obtained by using ideas from the BC method. Tataru's unique continuation theorem fails when the time-dependence of the metric or the coefficients is not real analytic \cite{Al,AB}, and therefore adaptations based on the BC method fails beyond this scenario. We refer the reader to \cite{E2} for recovery of coefficients when the time-dependence is real analytic. An alternative approach in deriving uniqueness results in the time-independent category started from the seminal work \cite{BK}, where Carleman estimates were used for the first time in the context of inverse problems. Proofs based on Carleman estimates tend to yield stronger stability estimates compared to BC method. Methods based on using the classical geometric optic solutions to the wave equation have also been quite fruitful in deriving uniqueness results in time-independent category (see for example \cite{BD,BJY,SU1,SU2}). 

In the time-dependent category, apart from \cite{E2} mentioned above, most of the results are concerned with wave equations with constant coefficient principal part. In \cite{St}, the author used geometric optic solutions for the wave equation with constant principal terms and an unknown zeroth order term to prove uniqueness by showing that the boundary data determines the light ray transform of the unknown scalar function in Minkowski space and subsequently inverting this transform. We refer the reader to \cite{A,I,Ki2,Ki4,RS,S} for similar results in this category.

Literature dealing with uniqueness results for the case of a wave equation with time-dependent first and zeroth order coefficients on a Riemannian manifold, where the time-dependence is non-analytic, is sparse. We refer to \cite{KiOk,W} for the study of recovering a time-dependent zeroth order term appearing in the wave equation. In the recent paper \cite{SY}, the authors used Fourier Integral Operators to show that a micro-local formulation of the Dirichlet to Neumann map $\Lambda_{\A,q}$ uniquely determines the light ray transforms of the one-form $\A$ and scalar function $q$. There, it was assumed that the coefficients are in some $\CI^k$ space with $k$ large enough. It was recently proved in \cite{FIKO} that if the one-form is $\CI^1$ smooth and the scalar function is continuous, then one can use the classical Gaussian beam construction to uniquely obtain the light ray transforms of $\A$ and $q$ from the knowledge of $\Lambda_{\A,q}$. The inversion of the light ray transform was also proved for the first time under the assumption that the geodesic ray transform is injective on the transversal manifold $(M,g)$ and that the coefficients are known for some explicit lengths of time near $t=0$ and $t=T$. In this paper, we generalize the result obtained in \cite{FIKO} to the case of non-smooth coefficients. We prove that if $(M,g)$ is simple, the Dirichlet to Neumann map uniquely determines the light ray transform of the non-smooth coefficients and subsequently show the inversion of the light ray transform as a Fourier Integral Operator under the additional assumption that the coefficients are known on a slightly larger set compared to the sharp domain $\mathcal D$ where no information can be obtained about the coefficients. This generalization and the difficulties therein are discussed in more detail in the subsequent section.

\subsection{Comments about our results}

We discuss some of the main novelties of our result, both by previewing some of the technical challenges and also by motivating the study of non-smooth coefficients in their own right. The technical difficulties are three-fold. One difficulty stems from the study of the forward problem and the need for sharper energy estimates for the determination of the correction terms appearing in the formal geometric optic ansatz. Another key difficulty stems from the one-form $\A$, as any lack of smoothness in the one-form appears at the level of the principal term corresponding to the geometric optic ansatz, thus making the task of a meaningful geometric optic solution to the wave equation and the reduction to the light ray transform of the coefficients more challenging. Finally, let us remark that in \cite{FIKO} the inversion of the light ray transform was proved for smooth coefficients. We generalize this inversion method to non-smooth functions by extending the notion of the light ray transform and the inversion method, in a distributional sense, to non-smooth coefficients.  

Aside from the technical challenges, it should be remarked that the recovery of non-smooth coefficients is a well-motivated question in its own right as it can be associated with the determination of various unstable phenomenon which can not be modeled by smooth parameters. For elliptic equations, this topic has received a lot of attention these last few decades (see \cite{AP,CR,H,HT,KU}). However, only few authors have addressed this issue for hyperbolic equations. Concerning the recovery of time dependent coefficients, \cite{HK} seems to be the only paper addressing this issue. The result of \cite{HK} concerns the recovery of a zeroth order coefficient on a flat Lorentzian manifold with the Minkowski metric. In Theorem \ref{t1} and \ref{t2}, we prove, for what seems to be the first time, the extension of this work to the recovery of non-smooth first and zeroth order coefficients appearing in a hyperbolic equation associated with a more general Lorentzian manifold.

Let us observe also that our inverse problem is intricately connected with the recovery of nonlinear terms appearing in hyperbolic equations. Indeed, following the strategy set by \cite{CK,CK2,I2,I3} for parabolic equations, through a linearization procedure initially introduced by \cite{I2}, one can reduce the problem of determining coefficients appearing in a non-linear problem to the recovery of time-dependent coefficients appearing in a linear equation.  In \cite{Ki5} the author proved the extension of this approach to semi-linear hyperbolic equations. Note that in this procedure, the time-dependent coefficient under consideration depends explicitly on solutions of the nonlinear equation. Therefore, following  the analysis of \cite{Ki5}, the recovery of non-smooth coefficients can be seen as an important step in the more difficult problem of determining quasi-linear terms appearing in nonlinear hyperbolic equations.

\subsection{Outline of the paper}
This paper is organized as follows. In Section~\ref{prelim}, we start by considering the direct problem \eqref{eq1} and rigorously justify the definition \eqref{dnmapspace} also deriving a key boundary integral identity (see Lemma~\ref{allesandrini}). Moreover, we discuss smooth approximations of the coefficients $\A, q$ and also extend the notion of the light ray transform to $L^{p}$ functions. Section~\ref{GOsection} is concerned with the construction of geometric optic solutions to \eqref{eq1} concentrating on maximal null geodesics in the set $\mathcal D$. In Section~\ref{uniqueness} we prove Theorems~\ref{t1}-\ref{t2} by applying the geometric optic construction and Proposition~\ref{lightray}. Finally, Section~\ref{lightray section} is concerned with the proof of statement (i) in Proposition~\ref{lightray}. As explained in Section~\ref{main} statement (ii) follows analogously to statement (ii) in \cite[Proposition 1.4]{FIKO}.
\section{Preliminaries}
\label{prelim}
\subsection{Direct problem}
\label{Forwardproblem}
In this section we study the wave equation \eqref{eq1} and show that for $\A,q$ satisfying \eqref{regularity} and each $f \in H^1_0((0,T]\times \pd M)$ it admits a unique solution $u$ in energy space \eqref{energyspace}. We will repeatedly use the Sobolev embedding theorem as follows. 
\bel{soem}
\|f_1f_2\|_{L^{p_1}(0,T;L^2(M))}\lesssim \|f_1\|_{L^{p_1}(0,T;L^{p_2}(M))}\|f_2\|_{\CI(0,T;H^1(M))}.
\ee
This estimate holds since $H^1(M)\subset L^{\frac{2n}{n-2}}(M)$ for $n>2$ and $H^1(M) \subset L^p(M)$ for $n=2$ and any $p\in[1,\infty)$. In order to study the IBVP given by \eqref{eq1}, we start by considering the following IBVP
\bel{eq2}\left\{ \begin{array}{rcll} -\Delta_{\g} v+\A\nabla^{\g}v+ qv =  F,&  (t,x) \in \M ,\\
v(0,x)=v_0(x),\ \ \partial_t v(0,x)=v_1(x),& x\in M\\
v(t,x)=0,  & (t,x)\in (0,T)\times \pd M.& \end{array}\right.
\ee
We have the following well-posedness result for this IBVP.
\begin{prop}\label{p1} Let $p_1\in(1,+\infty)$ and $p_2\in [n,+\infty)\setminus\{2\}$. For $q\in L^{p_1}(0,T;L^{p_2}(M))$, $\A\in L^\infty(\M; T^*\M)$ and $F\in L^{p_1}(0,T; L^2(M))$, problem \eqref{eq2} admits a unique solution $v$ in the space
 \bel{energyhom}
\mathscr{X}_0:=\mathcal C([0,T];H^1_0(M))\cap \mathcal C^1([0,T];L^2(M))
\ee 
satisfying $\partial_\nu v\in L^2((0,T)\times\partial M)$ and the estimate
\bel{p1a} \norm{\partial_\nu v}_{L^2((0,T)\times\partial M)}+\norm{v}_{\mathscr{X}_0}\leq C (\norm{v_0}_{H^1(M)}+\norm{v_1}_{L^2(M)}+\norm{F}_{L^{p_1}(0,T; L^2(M))}),\ee
with $C$ depending only on $p_1$, $p_2$, $n$, $T$, $M$ and any $N\geq \norm{q}_{L^{p_1}(0,T;L^{p_2}(M))}+\|\A\|_{L^\infty(\M)}$.
\end{prop}
\begin{proof}
We will prove this result by following the approach developed in \cite[Proposition 2.1]{HK}. Our first goal is to  show that for any $v\in W^{2,\infty}(0,T;H^1_0(M))$ solving \eqref{eq2}, the a priori estimate \eqref{p1a} holds true. Then, applying \cite[Theorem 8.1, Chapter 3]{LM1}, \cite[Remark 8.2, Chapter 3]{LM1} and \cite[Theorem 8.3, Chapter 3]{LM1}, the proof will be completed. We introduce the energy $E(t)$ at time $t\in[0,T]$ given by
$$E(t):=\int_M\left( |\partial_t v(t,x)|^2+|\nabla^g v(t,x)|^2\right)dV_g(x).$$
Multiplying \eqref{eq2} by $\overline{\partial_tv}$, taking the real part  and integrating by parts we get
\bel{p1b} \begin{aligned}E(t)=&-2\re\int_0^t\int_M[\A(s,x)\nabla^{\g}v(s,x)+ q(s,x)v(s,x)]\overline{\partial_t v(s,x)}dV_g(x)ds\\
\ &+2\re\int_0^t\int_M F(s,x)\overline{\partial_t v(s,x)}dV_g(x)ds.\end{aligned}\ee
Repeating the arguments of \cite[Proposition 2.1]{HK} we get
\bel{p1c}\begin{aligned}&\abs{\int_0^t\int_M q(s,x)v(s,x)\overline{\partial_t v(s,x)}dV_g(x)ds}+\abs{\int_0^t\int_M F(s,x)\overline{\partial_t v(s,x)}dV_g(x)ds}\\
&\leq  \norm{F}_{L^{p_1}(0,T; L^2(M))}^2+ C\left(\int_0^tE(s)^{\frac{p_1}{p_1-1}}ds\right)^{\frac{p_1-1}{p_1}},\end{aligned}\ee
where $C$ depends only on $T$, $M$, $p_1$, $p_2$, $n$ and any $N\geq \norm{q}_{L^{p_1}(0,T;L^{p_2}(M))}$. 
In the same way, we obtain 
\bel{p1d}\begin{aligned}&\abs{\int_0^t\int_M[\A(s,x)\nabla^{\g}v(s,x)]\overline{\partial_t v(s,x)}dV_g(x)ds}\\
&\leq \|\A\|_{L^\infty(\M)}\int_0^tE(s)ds\\
&\leq \|\A\|_{L^\infty(\M)} t^{\frac{1}{p_1}}\left(\int_0^tE(s)^{\frac{p_1}{p_1-1}}ds\right)^{\frac{p_1-1}{p_1}}\\
&\leq \|\A\|_{L^\infty(\M)} T^{\frac{1}{p_1}}\left(\int_0^tE(s)^{\frac{p_1}{p_1-1}}ds\right)^{\frac{p_1-1}{p_1}}.\end{aligned}
\ee
Combining \eqref{p1c}-\eqref{p1d}  with \eqref{p1b}, we obtain 
$$ E(t)\leq  \norm{F}_{L^{p_1}(0,T; L^2(M))}^2+ C\left(\int_0^tE(s)^{\frac{p_1}{p_1-1}}ds\right)^{\frac{p_1-1}{p_1}},$$
where $C$ depends only on $T$, $M$, $p_1$, $p_2$, $n$ and any $N\geq \norm{q}_{L^{p_1}(0,T;L^{p_2}(M))}+\|\A\|_{L^\infty(\M)}$. 
Using this last estimate we can deduce that \eqref{eq2} admits a unique solution $v$ in the space
 \eqref{energyhom} satisfying 
\bel{p1e}\norm{v}_{\mathscr{X}_0}\leq C (\norm{v_0}_{H^1(M)}+\norm{v_1}_{L^2(M)}+\norm{F}_{L^{p_1}(0,T; L^2(M))})\ee
by applying arguments similar to the end of the proof of \cite[Proposition 2.1]{HK}. Therefore the proof of the proposition will be completed if we show that $\partial_\nu v\in L^2((0,T)\times\partial M)$ and that the estimate
\bel{p1f} \norm{\partial_\nu v}_{L^2((0,T)\times\partial M)}\leq C (\norm{v_0}_{H^1(M)}+\norm{v_1}_{L^2(M)}+\norm{F}_{L^{p_1}(0,T; L^2(M))})\ee
is fulfilled. For this purpose, notice that $v$ solves 
$$\left\{ \begin{array}{rcll} -\Delta_{\g} v(t,x)= F_v(t,x),&  (t,x) \in \M ,\\
v(0,x)=v_0(x),\ \ \partial_t v(0,x)=v_1(x),& x\in M\\
v(t,x)=0,  & (t,x)\in (0,T)\times \pd M,& \end{array}\right.$$
with $F_v=-\A\nabla^{\g}v-qv+F$. Applying the Sobolev embedding theorem we deduce that $F_v\in L^1(0,T;L^2(M))$. Then, applying \cite[Lemma 2.39]{KKL} we deduce that $\partial_\nu v\in L^2((0,T)\times\partial M)$ and
$$ \begin{aligned}\norm{\partial_\nu v}_{L^2((0,T)\times\partial M)}&\leq C(\norm{F_v}_{L^1(0,T;L^2(M))}+\norm{v_0}_{H^1(M)}+\norm{v_1}_{L^2(M)})\\
\ &\leq C(\norm{v}_{\mathcal C([0,T];H^1(M))}+\norm{v}_{\mathcal C^1([0,T];L^2(M))}+\norm{F}_{L^{p_1}(0,T; L^2(M))}).\end{aligned}$$
Combining this with \eqref{p1e} we deduce \eqref{p1f} and this completes the proof of the proposition.

\end{proof}

We can use Proposition~\ref{energyhom} to show that equation \eqref{eq1} admits a unique solution $u$ in energy space \eqref{energyspace}. Recall the following classical IBVP:
\bel{eq3}\left\{ \begin{array}{rcll} &-\Delta_{\g} w=  0,&  (t,x) \in \M ,\\
&w(0,x)=0,\ \ \partial_t w(0,x)=0,& x\in M\\
&w(t,x)=f,  & (t,x)\in (0,T)\times \pd M.& \end{array}\right.
\ee
According to \cite[Theorem 2.30]{KKL}  (see also \cite{LLT}), this equation admits a unique solution $w$ in the energy space \eqref{energyspace}. 
We now return to \eqref{eq1} and note that we have $u=w+v$, where $w$ solves \eqref{eq3} with boundary term $f$, and $v$ solves \eqref{eq2} with $F:=-\A\nabla^{\g}w-qw$. As $\A,q$ satisfy \eqref{regularity} and since $w$ is in the energy space \eqref{energyspace}, it is immediate that $F \in L^{p_1}(0,T;L^2(M))$. Thus, Proposition~\ref{p1} applies to show that $u$ is in the energy space \eqref{energyspace}, with $\partial_\nu u\in L^2((0,T)\times\partial M)$, and we have that 
$$ \norm{\partial_\nu u}_{L^2((0,T)\times\partial M)}+\|u\|_{\mathscr X} \leq C \|f\|_{H^1_0((0,T]\times \pd M)}.$$ 
Using this estimate we can define the DN map as the bounded operator from $H^1_0((0,T]\times \pd M)$ to $L^2((0,T)\times\partial M)$ defined by
$$\Lambda_{\A,q} f=\left(\pd_{\bar\nu} u - \frac{\A\bar{\nu}}{2}u\right)|_{(0,T)\times \pd M}$$
for $u$ the solution of \eqref{eq1}.

We have the following lemma that will be used in Section~\ref{GOsection}.
\begin{lem}
\label{integral estimate}
Let $F \in L^{p_1}(0,T;L^2(M))$ and suppose $u$ is the unique solution to \eqref{eq2} subject to $u_0=u_1=0$. Then the following estimate holds:
\[
\|u\|_{\CI(0,T;L^2(M))} \leq C  \norm{\int_0^t F(s) \,ds}_{L^{p_1}(0,T;L^2(M))}.
\]
\end{lem}
\begin{proof}
We set $v(t,x):=\int_0^t u(s,x)\,ds$ and note that $v$ solves
\bel{eqqgo2}\left\{ \begin{array}{rcll} -\Delta_{\g} v =  H,&  (t,x) \in \M ,\\
v(0,x)=0,\ \ \partial_t v(0,x)=0,& x\in M\\
 v(t,x)=0, & (t,x) \in (0,T)\times \pd M,& \end{array}\right.
\ee
with
$$H:=-\int_0^t\A(s,x)\nabla^{\g}u(s,x)\,ds-\int_0^tq(s,x)u(s,x)ds+\int_0^tF(s,x)\,ds.$$
Since $u$ is in the energy space \eqref{energyspace}, we deduce $v\in \mathcal C^2([0,T];L^2(M))\cap \mathcal C^1([0,T];H^1(M))$. In addition, since $qu\in L^{p_1}(0,T;L^2(M))$ (see \eqref{soem}) and $\A\in L^\infty(\M;T^*\M)$, we deduce that $H\in W^{1,p_1}(0,T;L^2(M))\subset L^2(\M)$ and that $v$ solves the elliptic boundary value problem
$$\left\{ \begin{array}{rcll} -\Delta_g v =  E,&  (t,x) \in (0,T)\times M ,\\
v(t,x)=0, & (t,x) \in (0,T)\times\partial M,& \end{array}\right.$$
with $E=-\partial_t^2v+H\in L^2(\M)$. Then, from the elliptic regularity of solutions of this boundary value problem, we get $v\in L^2(0,T;H^2(M))$ and it follows that $v\in H^2(\M)$. We define the energy $E(t)$ at time $t$ associated with $v$ and given by
\bel{l1f} E(t):=\int_M \left(|\partial_tv|^2(t,x)+|\nabla^gv|_g^2(t,x)\right)\,dV_g(x)\geq \int_M |u|^2\,dV_g(x).$$
Multiplying \eqref{eqqgo2} by $\overline{\partial_tv}$ and taking the real part, we find
$$\begin{aligned}E(t)=&-2\re \left(\int_0^t\int_M \left(\int_0^sq(\tau,x)u(\tau,x)\,d\tau\right)\overline{\partial_tv(s,x)}\,dV_g(x)\,ds\right)\\
&-2\re \left(\int_0^t\int_M \left(\int_0^s\A(\tau,x)\nabla^{\g}u(\tau,x)\,d\tau\right)\overline{\partial_tv(s,x)}\,dV_g(x)\,ds\right)\\
&+2\re \left(\int_0^t\int_M \left(\int_0^sF(\tau,x)\,d\tau\right)\overline{\partial_tv(s,x)}\,dV_g(x)\,ds\right).\end{aligned}\ee
Repeating some arguments of \cite[Lemma 3.1]{HK}, we find
\bel{l1a}\begin{aligned}&\abs{\int_0^t\int_M \left(\int_0^sq(\tau,x)u(\tau,x)\,d\tau\right)\overline{\partial_tv(s,x)}\,dV_g(x)\,ds}\\
&\leq C\norm{q}_{L^{p_1}(0,T; L^{p_2}(M))}^2\left(\int_0^t E(\tau)^{\frac{p_1}{(p_1-1)}}\,d\tau\right)^{\frac{(p_1-1)}{p_1}}+\frac{E(t)}{5},\end{aligned}\ee
\bel{l1b}\begin{aligned}&\abs{\int_0^t\int_M \left(\int_0^sF(\tau,x)\,d\tau\right)\overline{\partial_tv(s,x)}\,dV_g(x)\,ds}\\
&\leq \norm{F_{*}}_{L^{p_1}(0,T;L^2(M))}^2+T^{{\frac{p_1-1}{p_1}}}\left(\int_0^t E(\tau)^{\frac{p_1}{p_1-1}}\,d\tau\right)^{{\frac{p_1-1}{p_1}}},\end{aligned}\ee
where $F_{*}(t,x):=\int_0^tF(s,x)ds$ and $C>0$ depends on $M$ and $T$. In the same way, using the fact that $ \div_{\g} \A \in L^{p_1}(0,T;L^{p_2} (M))$ and $v\in \mathcal C^1([0,T];H^1_0(M))$, we get
\bel{l1c}\begin{aligned} &\int_0^t\int_M \left(\int_0^s\A\nabla^{\g}u(\tau,x)d\tau\right)\overline{\partial_tv(s,x)}\,dV_g(x)\,ds\\
&=-\int_0^t\int_0^s\int_M (\div_{\g} \A)u\,\overline{\partial_tv(s,x)}\,dV_g(x)\,d\tau\, ds-\int_0^t\int_M \int_0^su\,\A(\tau,x)\overline{\partial_t\nabla^{\g}v(s,x)}\,dV_g(x)\,d\tau\, ds\\
&\ \ \ \ \ -\int_0^t\int_M b(s,x)|\partial_tv(s,x)|^2\,dV_g(x)\, ds.\end{aligned}\ee
Repeating the arguments of \eqref{l1a}, we find
\bel{l1d}
\begin{aligned}
&\abs{\int_0^t\int_0^s\int_M (\div_{\g} \A)u(\tau,x)\overline{\partial_tv(s,x)}\,dV_g(x)\,d\tau\, ds+\int_0^t\int_M b(s,x)|\partial_tv(s,x)|^2\,dV_g(x)\, ds} \\
&\leq C\left(\int_0^t E(\tau)^{\frac{p_1}{(p_1-1)}}\,d\tau\right)^{\frac{(p_1-1)}{p_1}}+\frac{E(t)}{5},
\end{aligned}
\ee
with $C$ depending on $T$, $M$, $\norm{\div_{\g} \A}_{L^{p_1}(0,T; L^{p_2}(M))}$ and $\norm{ b}_{L^{\infty}(\M))}$. Moreover, applying Fubini's theorem, we have
$$\begin{aligned}&\int_0^t\int_M \int_0^s u(\tau,x)\A(\tau,x)\overline{\partial_t\nabla^{\g}v(s,x)}\,dV_g(x)\,d\tau\, ds\\
&=\int_M \int_0^tu\A(\tau,x)\left(\int_\tau^t\overline{\partial_t\nabla^{\g}v(s,x)}\,ds\right)\,d\tau \,dV_g(x)\\
&=\int_M \int_0^tu\A(\tau,x)\overline{\nabla^{\g}v(t,x)}\,dV_g(x)\,d\tau-\int_M \int_0^tu\A(\tau,x)\overline{\nabla^{\g}v(\tau,x)}\,d\tau \,dV_g(x) .\end{aligned}$$
It follows that
$$\begin{aligned}&\abs{\int_0^t\int_M \int_0^su\A(\tau,x)\overline{\partial_t\nabla^{\g}v(s,x)}\,dV_g(x)\,d\tau\, ds}\\
&\leq  \norm{\A}_{L^\infty(\M)} \left(\int_0^tE(\tau)^{\frac{1}{2}}\,d\tau\right) E(t)^{\frac{1}{2}}+\norm{\A}_{L^\infty(\M)}\int_0^tE(\tau)\,d\tau\\
&\leq  5\norm{\A}_{L^\infty(\M)}^2 \left(\int_0^tE(\tau)^{\frac{1}{2}}\,d\tau\right)^2+\frac{E(t)}{5}+\norm{\A}_{L^\infty(\M)}\int_0^tE(\tau)\,d\tau\\
 &\leq  5\norm{\A}_{L^\infty(\M)}^2 T \left(\int_0^tE(\tau)\,d\tau\right)+\frac{E(t)}{5}+\norm{\A}_{L^\infty(\M)}\int_0^tE(\tau)\,d\tau\\
 &\leq  (5\norm{\A}_{L^\infty(\M)}^2 T+\norm{\A}_{L^\infty(\M)}) \left(\int_0^tE(\tau)\,d\tau\right)+\frac{E(t)}{5}.\end{aligned}$$
Applying H\"older's inequality, we get
$$\begin{aligned}&\abs{\int_0^t\int_M \int_0^su\A(\tau,x)\overline{\partial_t\nabla^{\g}v(s,x)}\,dV_g(x)\,d\tau \,ds}\\
 &\leq  (5\norm{\A}_{L^\infty(\M)}^2 T+\norm{\A}_{L^\infty(\M)}) T^{\frac{1}{p_1}}\left(\int_0^tE(\tau)^{\frac{p_1}{(p_1-1)}}\,d\tau\right)^{\frac{(p_1-1)}{p_1}}+\frac{E(t)}{5}.\end{aligned}$$
Combining this with \eqref{l1d}, we deduce that
\bel{l1e}\abs{\int_0^t\int_M \int_0^su\A(\tau,x)\overline{\partial_t\nabla^{\g}v(s,x)}\,dV_g(x)\,d\tau \,ds}\leq C\left(\int_0^tE(\tau)^{\frac{p_1}{(p_1-1)}}\,d\tau\right)^{\frac{(p_1-1)}{p_1}}+\frac{2E(t)}{5},\ee
with $C$ depending only on $T$ and $\norm{\A}_{L^\infty(\M)}$. We deduce that there exists $C$ depending on $T$, $M$, $\norm{q}_{L^{p_1}(0,T; L^{p_2}(M))}$, $\norm{\A}_{L^\infty(\M)}$ and $\norm{\div_{\g} \A}_{L^{p_1}(0,T; L^{p_2}(M))}$ such that
$$E(t)\leq \frac{4E(t)}{5}+C\left(\int_0^t E(\tau)^{\frac{p_1}{p_1-1}}\,d\tau\right)^{{\frac{p_1-1}{p_1}}}+\norm{F_{*}}_{L^{p_1}(0,T;L^2(M))}^2$$
and therefore
$$E(t)\leq 5C\left(\int_0^t E(\tau)^{\frac{p_1}{p_1-1}}\,d\tau\right)^{{\frac{p_1-1}{p_1}}}+5\norm{F_{*}}_{L^{p_1}(0,T;L^2(M))}^2.$$
Applying the Gronwall inequality yields
$$E(t)^{\frac{p_1}{p_1-1}}\leq c_1\norm{F_{*}}_{L^{p_1}(0,T;L^2(M))}^{\frac{2p_1}{p_1-1}}e^{c_2t}\leq c_1\norm{F_{*}}_{L^{p_1}(0,T;L^2(M))}^{\frac{2p_1}{p_1-1}}e^{c_2T},\quad t\in(0,T),$$
where $c_1$ depends only on $p_1$ and $c_2$ on $C$ and $p_1$.
\end{proof}
\subsection{Dirichlet to Neumann map}
\label{DNformulation}
In this section we will derive a representation formula involving the DN map (Lemma~\ref{allesandrini}) and also recall some invariance properties of the DN map (Lemma~\ref{DNgauge}).

Let us consider the following problem 
\bel{eq4}
\left\{ \begin{array}{ll} L_{\A,q}^*v= -\Delta_{\g} v - \A\nabla^{\g} v+ (q-\div_{\g} \A) v  =  0, & \mbox{on}\ \M,\\  
v  =  h, & \mbox{on}\ (0,T)\times \pd M,\\  
 v(T,\cdot)  =0,\quad  \pd_tv(T,\cdot)  =0 & \mbox{on}\ M.
\end{array} \right.
\ee 
Here, the differential operator $L_{\A,q}^*$ represents the formal adjoint of $L_{\A,q}$. Repeating the arguments of the previous section we can prove that, for each $h \in H^1_0([0,T)\times \pd M)$, this problem admits a unique solution $v$ in energy space \eqref{energyspace}, with $\partial_{\bar\nu} v\in L^2((0,T)\times\partial M)$, satisfying the estimate
$$ \norm{\partial_{\bar\nu} v}_{L^2((0,T)\times\partial M)}+\|v\|_{\mathscr X} \leq C \|h\|_{H^1_0([0,T)\times \pd M)}.$$ 
Therefore, we can define the DN map associated with \eqref{eq4} as follows
\bel{DNadjoint}
\Lambda^*_{\A,q} h=\left(\pd_{\bar\nu} v +\frac{\A\bar\nu}{2}v\right)|_{(0,T)\times \pd M}.
\ee
It is straightforward to show that
\bel{adjoint1}   
<\Lambda_{\A,q}f,h>=<f,\Lambda_{\A,q}^*h> \quad \forall (f,h) \in H^1_0((0,T]\times \pd M)\times H^1_0([0,T)\times \pd M),
\ee
where $\langle f_1,f_2\rangle:=\int_{(0,T)\times\pd M}f_1f_2\,dV_{\g}$. Using this equality together with Green's identity, we can derive the following classical representation formula.
\begin{lem}
\label{allesandrini}
Let $\A_1,\A_2,q_1,q_2$ satisfy \eqref{regularity}. Given any $f_1 \in H^1_0((0,T]\times \pd M)$ and $f_2 \in H^1_0([0,T)\times \pd M)$, the following identity holds:
\bel{int-iden}
\langle (\Lambda_{\A_1,q_1}-\Lambda_{\A_2,q_2})f_1,f_2\rangle=\int_\M \left[\frac{u_2 \A\nabla^{\g}u_1-u_1 \A\nabla^{\g}u_2}{2}+(q-\frac{1}{2}\div_{\g}\A)u_1u_2\right]\,dV_{\g},
\ee
where $\A:=\A_1-\A_2$, $q:=q_1-q_2$, $u_1$ solves \eqref{eq1} with $\A=\A_1$, $q=q_1$ and lateral boundary term $f_1$ while $u_2$ solves \eqref{eq4} with $\A=\A_2$, $q=q_2$ and lateral boundary term $f_2$.
\end{lem}

We also have the following lemma regarding the gauge equivalence of the Dirichlet to Neumann map:

\begin{lem}
\label{DNgauge}
Let $\A,q$ satisfy \eqref{regularity}. Suppose $\psi \in \CI^1(\M)$ vanishes on $(0,T)\times \pd M$ and satisfies $\Delta_{\g}\psi \in L^{p_1}(0,T;L^{p_2}(M))$. Then 
$$\Lambda_{\A,q}=\Lambda_{\tilde{\A},\tilde{q}},$$ 
where
\bel{gauge}
\tilde{\A}=\A+\bar{d}\psi \quad \text{and}\quad \tilde{q}=q+\frac{1}{2}\Delta_{\g}\psi-\frac{1}{2}\A\nabla^{\g}\psi-\frac{1}{4}\langle \nabla^{\g}\psi,\nabla^{\g}\psi\rangle_{\g}.
\ee 
\end{lem}

\begin{proof}
We start by observing that if $u$ solves differential equation \eqref{eq1} with coefficients $\A, q$ and a lateral boundary condition $f$, then $\tilde{u}=e^{\frac{1}{2}\psi}u$ solves equation \eqref{eq1} with coefficients $\tilde{\A}, \tilde{q}$ and the same lateral boundary condition $f$. Then it follows that
$$\begin{aligned}\Lambda_{\tilde{\A},\tilde{q}}f&=\left(\pd_{\bar\nu} \tilde{u} - \frac{\tilde{\A}\bar\nu}{2}\tilde{u}\right)|_{(0,T)\times \pd M}=\left(\pd_{\bar\nu}u+\frac{\partial_{\bar\nu} \psi}{2}u - \frac{\tilde{\A}\bar\nu}{2}u\right)|_{(0,T)\times \pd M}\\
\ &=\left(\pd_{\bar\nu}u- \frac{\A\bar\nu}{2}u\right)|_{(0,T)\times \pd M}=\Lambda_{\A,q}f.\end{aligned}$$
\end{proof}

\subsection{Smooth approximation of the coefficients $\A$ and $q$}
\label{smoothapp}

The goal of this section is to show that given one-forms $\A_{k}, k=1,2$ satisfying \eqref{regularity}, it is possible to find smooth approximations $\A_{k,\rho}$ that are defined in a slightly larger manifold $\hat{\M}$ and such that \eqref{mollified} holds. Let $M\subset \hat{M}^\inter\subset \tilde{M}^\inter$ denote a small artificial extension of the simple manifold $M$, so that $\hat{M},\tilde{M}$ are also simple manifolds and define $\hat{\M}=\R\times \hat{M}$. We first consider the Sobolev extension of $\A_k, k=1,2$ to the larger manifold $\R\times M$ such that the extension belongs to $W^{1,1}(\R;L^{2}(M)) \cap \CI(\R\times M)$ and then extend this extended one-form to $\R\times \tilde{M}$ by setting it equal to zero on $\R\times (\tilde{M}\setminus M)$. The scalar functions $q_k$ are extended to $\R\times \tilde{M}$ by setting them equal to zero outside of $\M$. Let $p \in \tilde{M}\setminus \hat{M}$. As $\tilde{M}$ is simple, there exists a global coordinate chart on a neighborhood of $\hat{M}$ given by $(y^1,\ldots,y^n)$. Indeed one such coordinate system would be the polar normal coordinates around a point $p \in \tilde{M}\setminus \hat{M}$ (see for example  \cite[Chaper 9, Lemma 15]{Sp}). We then consider the coordinate chart $(t,y^1,\ldots,y^n)$ on a neighborhood of $\hat{\M}$ in $\R \times \tilde{M}$ and note that using this chart we can easily define smooth approximations of the coefficients $\A_k,q_k$. Indeed, let $\rho>0$ and define the smooth function $\zeta_\rho:\hat{M}\to \R$ through $$\zeta_{\rho}(t,y)=\rho^{\frac{n+1}{4}}\chi(\rho^{\frac{1}{4}}\sqrt{t^2+(y^1)^2+\ldots+(y^n)^2})$$ where $\chi:\mathbb R\to \mathbb R$ is a non-negative smooth function satisfying $\chi(t)=1$ for $|t|<\frac{1}{4}$ and $\chi=0$ for $|t|>\frac{1}{2}$ and $\|\chi\|_{L^1(\R)}=1$. We define the smooth approximations $\A_{k,\rho}$ of the coefficients $\A_k$ through the expressions
\bel{smooth app}
 \A_{k,\rho}(t,x):=(\A_k*\zeta_\rho)(t,x)=b_{k,\rho}\,dt+A_{k,\rho} \quad \forall (t,x) \in \hat{\M} \quad k=1,2.
\ee 
and note that in view of \eqref{regularity}, the following estimates hold for $k=1,2$:
\bel{mollified}
\begin{aligned}
&\lim_{\rho\to \infty}\left(\|\A_{k,\rho}-\A_k\|_{W^{1,1}(0,T;L^{2}(M))}+\|\A_{k,\rho}-\A_k\|_{L^p(\M)}\right)=0\quad \forall \, p\in [1,\infty),\\
&\|\A_{k,\rho}\|_{W^{k,\infty}(\M)}\lesssim \rho^{\frac{k}{4}} \quad \forall k \in \N^*.
\end{aligned}
\ee
Additionally, since $\A \in \CI(\M;T^*\M)$ we can write
\bel{cont}
\lim_{\rho \to \infty} \|\A_{k,\rho}-\A\|_{\CI((0,T)\times \Omega_\rho)} =0,\quad k=1,2
\ee
where $\Omega_{\rho}=\{x \in \tilde{M}\,|\, \dis{(x,\pd M)}\gtrsim \rho^{-\frac{1}{4}}\}.$
\subsection{Light ray transform as a Fourier Integral Operator}
\label{FIO}

The main goal of this section is to extend the notion of $\mathcal L_{\beta}$ over scalar functions in $L^p$ Sobolev spaces. This extension is based on showing that $\mathcal L$ is a Fourier Integral Operator. We will assume through out this section that $(\M ,\g)$ and $(M,g)$ are as discussed in the introduction and that $M\subset \hat{M}^\inter$ with $\hat{M}$ as in Section~\ref{smoothapp}. We start with the notion of light ray transform $\mathcal L$ of scalar functions over null geodesics in $\R\times\hat{M}$ showing that it has a unique continuous extension as an operator from $\E'(\R \times \hat{M})$ to $\D'(\R \times \pd_- S\hat{M})$. This would naturally show that the light ray transform $\mathcal L_{\beta}$ of scalar functions over null geodesics on $\M$ has a continuous extension from $L^{1}(0,T;L^{2}(M))$ to  $\D'(\R\times \pd_-SM)$ as $L^{1}(0,T;L^{2}(M))\subset \E'(\R \times \hat{M})$.

We will now show that the kernel of $\mathcal L$ is locally represented by an oscillatory integral. It suffices to consider $f \in \CI_c^\infty(\R \times \hat{M})$ that is supported in a coordinate neighborhood and work in local coordinates on $\hat{M}$. Let us also extend the geodesics $\gamma(\cdot; y,v)$, $(y,v) \in \pd_- S\hat{M}$, as functions from $\R$ to $\hat{M}$ so that $\gamma(s;y,v) \notin \supp(f)$ for $s \notin [0, \tau_+(x,v)]$.
Then in local coordinates
    \begin{align*}
\mathcal L f(s,y,v) 
&= 
\int_\R f(r+s,\gamma(r;y,v))\, dr
= 
\int_\R \int_{\R^n} f(t, x) \delta(x-\gamma(t-s; y,v)) \,dx\, dt,
    \end{align*}
and writing $\phi(x,t;s,y,v;\xi) = \xi(x-\gamma(t-s; y,v))$ it holds that 
$$
\delta(x-\gamma(t-s; y,v))
= \int_{\R^n} e^{i \phi(x,t;s,y,v;\xi)} \,d\xi.
$$
Moreover, $\phi$ is an operator phase function in the sense of \cite[Def. 1.4.4]{FIO1}.
Indeed for fixed $(s,y,v)$ it clearly has no critical points when $\xi \ne 0$. That the same is true for fixed $(t,x)$ follows from the next lemma. 

\begin{lem}
Let $(y_0,v_0) \in \pd_- S\hat{M}$ and $r_0 \in (0, \tau_+(y_0,v_0))$, and consider a small neighborhood $U$ of $(r_0, y_0, v_0)$ in $\R \times \pd_- S\hat{M}$. Then $\gamma(r;y,v)$ as a map from $U$ to $\hat{M}$ has surjective differential at $(r_0, y_0, v_0)$. 
\end{lem}
\begin{proof}
Write $x_0 = \gamma(r_0; y_0, v_0)$, $w_0 = \dot \gamma(r_0; y_0, v_0)$ and let $\xi_* \in T_{x_0} \hat{M}$. Choose a path $\alpha$ in $\hat{M}$ such that $\alpha(0) = x_0$ and $\dot \alpha(0) = \xi_*$. Consider $-w_0$ in local coordinates as a vector in all $T_{\alpha(\epsilon)} \hat{M}$ for small $\epsilon > 0$. As
$$
\gamma(r_0; x_0, -w_0) = y_0, 
\quad 
\dot \gamma(r_0; x_0, -w_0) = -v_0 \notin T_{y_0}(\pd\hat{M}),
$$
it follows from the implicit function theorem that there is unique $r(\epsilon)$ near $r_0$ such that $\gamma(r(\epsilon); \alpha(\epsilon), -w_0) \in \pd \hat{M}$.
Writing 
$$
y(\epsilon) = \gamma(r(\epsilon); \alpha(\epsilon), -w_0), 
\quad 
v(\epsilon) = -\dot \gamma(r(\epsilon); \alpha(\epsilon), -w_0),
$$
we have that $\gamma(r(\epsilon); y(\epsilon), v(\epsilon)) = \alpha(\epsilon)$.
Hence the differential of the map $\gamma$ takes vectors $(\dot r(0), \dot y(0), \dot v(0))$ to $\xi_* = \dot \alpha(0)$.
\end{proof}
As $\phi$ is an operator phase function, 
the light ray transform $\mathcal L$ has a unique continuous extension as an operator from $\E'(\R \times \hat{M})$ to $\D'(\R \times \pd_- S\hat{M})$ by
\cite[Th. 1.4.1]{FIO1}.

\section{Geometric Optics}
\label{GOsection}
Throughout this section we consider one-forms $\A_1,\A_2$ and scalar functions $q_1,q_2$ to satisfy regularity conditions given by \eqref{regularity} and consider their extensions to the manifold $\tilde{\M}$ and their smooth approximations on the manifold $\hat{\M}$ as outlined in Section~\ref{smoothapp}. We consider a fixed null geodesic $\beta \subset \mathcal D$ parametrized with respect to the time variable. The projection of this null-geodesic on $M$ is denoted by the (Riemannian) unit speed geodesic $\gamma(\cdot;y,v)$ defined over its maximal domain $I=[0,\tau_+(y,v)]$. We extend $\gamma$ to $\tilde{M}$ and let the interval $\hat{I}=[-\hat{\delta}_-,\tau_+(y,v)+\hat{\delta}_+]$ to denote the maximal domain of definition of $\gamma$ on the manifold $\hat{M}$. Subsequently we can parametrize the extended null geodesic $\beta$ on $\hat{\M}$ through
$$\beta(t;s,y,v)=(s+t,\gamma(r;y,v)) \quad \text{for}\quad  t\in\hat{I},$$
where $s \in \R$ is a constant. We are interested in constructing the so called geometric optic solutions $u_1, u_2$ in energy space \eqref{energyspace}, of the problems
\bel{Gsol}
\begin{array}{ll}
\left\{\begin{array}{l}
-\Delta_{\g}u_1+\A_1\nabla^{\g}u_1+q_1u_1=0,\quad (t,x)\in \M,\\ 
u_1(0,x)=\partial_tu_1(0,x)=0,\quad  x\in M,\end{array}\right.
\\
\\
\left\{\begin{array}{l}
-\Delta_{\g}u_2-\A_2\nabla^{\g}u_2+(q_2-\div_{\g}\A_2)u_2=0,\quad (t,x)\in \M,\\ 
u_2(T,x)=\partial_tu_2(T,x)=0,\quad  x\in M,
\end{array}\right.
\end{array}
\ee
taking the form
\bel{go1}u_1(t,x)=e^{ i\rho\Phi(t,x)}c_{1,\rho}(t,x)+R_{1,\rho}(t,x),\quad (t,x)\in \M,\ee
\bel{go2}u_2(t,x)=e^{-i\rho \Phi(t,x)}c_{2,\rho}(t,x)+R_{2,\rho}(t,x),\quad (t,x)\in \M,\ee
with $\rho>1$. The phase function $\Phi$ and the smooth amplitude functions $c_{j,\rho}, j=1,2$ are constructed in a way that the principal terms $e^{i\rho\Phi}c_{j,\rho}$ are compactly supported near the null geodesic $\beta$. The remainder terms $R_{j,\rho}$ asymptotically converge to zero as $\rho \to \infty$. 

As we are interested in a particular null geodesic $\beta$, we outline a polar normal coordinate system specific to this null geodesic. We start by considering a point $p$ on $\{0\}\times \gamma$ with $p \in \{0\}\times (\tilde{M}\setminus \hat{M})$ and construct the polar normal coordinates $(t,r,\theta)$ about the point $p$ defined for $r>0$ and $\theta \in S_p \tilde{M}=\{ v \in T_p \tilde{M}\,|\, |v|_g=1\}$ through the diffeomorphism $(t,x)=(t,\exp(r\theta))$. In this coordinate system the metric $\g$ is smooth away from the point $p$ and takes the form
\bel{}
\g(t,r,\theta)=-(dt)^2+(dr)^2+g_0(r,\theta),
\ee
where $g_0$ is a Riemannian metric on $S_p\tilde{M}$. As we will be only considering this coordinate system on the manifold $\hat{\M}$ and owing to the fact that $\hat{M}$ is simple, we can identify $\theta$ with a globally defined coordinate system $(\theta^1,\ldots,\theta^{n-2}) \in \R^{n-2}$. This can in fact be done in such a manner that the null geodesic $\beta$ on $\hat{\M}$ can be represented with coordinates $(s+s_0,s,0,\ldots,0)$ with $s \in \hat{I}$.

In order to make the analysis simpler, we will introduce a new coordinate system near $\beta$ denoted by $(z^0,z^1,\ldots,z^n)$ in terms of the polar normal coordinates $(t,r,\theta)$ on $\hat{\M}$ given by
\begin{itemize}
\item[(i)]{$z^0:=\frac{1}{\sqrt{2}}(t+r)$,}
\item[(ii)]{$z^1:=\frac{1}{\sqrt{2}}(-t+r+s_0)$,}
\item[(iii)]{$z^j:=\theta^j$ for $j=2,\ldots,n$.}
\end{itemize}
In this coordinate system, the null geodesic $\beta$ on $\hat{\M}$ is given by the coordinates $(s,0)$ with $s \in (a_0,b_0)$ for some constants $a_0,b_0$. Furthermore the metric $\g$ takes the form
\bel{metric}
\g(z)=2\,dz^0\,dz^1+\sum_{j,k=2}^n g_{jk}(z)\,dz^j\,dz^k.
\ee
We define a tubular neighborhood around the null geodesic $\beta$ where the amplitude functions are compactly supported, as follows:
\bel{tubular}
\mathcal V_{\beta} = \{z \in \hat{\M}\,|\, z^0 \in [a_0,b_0],\quad |z'|:=\sqrt{|z^1|^2+\ldots+|z^n|^2}<\delta'\},
\ee 
where $\delta'>0$ is sufficiently small that the set $\mathcal V_\beta$ is disjoint from $\{0\}\times M$ and $\{T\}\times M$. This can be guaranteed due to the assumption $\beta \subset \mathcal D$. 
\subsection{Construction of the Geometric Optics}
We proceed to carry out the construction of the geometric optic solutions to \eqref{Gsol} in detail. We impose to the remainder term 
$$R_{k,\rho}\in  \mathcal C([0,T];H^1_0( M))\cap \mathcal C^1([0,T];L^2( M)),\quad k=1,2$$ 
the following decay property
\bel{GO2} \lim_{\rho\to+\infty}(\norm{R_{k,\rho}}_{\CI(0,T;L^2(M))}+\rho^{-1}\norm{R_{k,\rho}}_{H^1(\M)})=0.\ee
To prove the decay of $R_{k,\rho}$ with respect to $\rho$, given by \eqref{GO2}, we need to suitably construct $\Phi$, $c_{1,\rho}$, $c_{2,\rho}$. We write
\bel{conjugated}
\begin{aligned}
&L_{\A_{1,\rho},q_1}(e^{i\rho\Phi}c_{1,\rho})=e^{i\rho\Phi}\left(\rho^2\mathcal S \Phi -i\rho\mathcal T_{\A_{1,\rho}} c_{1,\rho}+L_{\A_{1,\rho},q_1}c_{1,\rho}\right),\\
&L^*_{\A_{2,\rho},q_2}(e^{-i\rho \Phi}c_{2,\rho})=e^{-i\rho\Phi}\left(\rho^2\mathcal S \Phi +i\rho\mathcal T_{-\A_{2,\rho}} c_{1,\rho}+L^*_{\A_{2,\rho},q_2}c_{2,\rho}\right),
\end{aligned}
\ee
where 
\bel{eikonal-transport}
\mathcal S\Phi:=\langle \nabla^{\g}\Phi,\nabla^{\g}\Phi\rangle_{\g}\quad \text{and}\quad \mathcal T_{\A}\cdot=2\langle\nabla^{\g}\Phi,\nabla^{\g}\cdot\rangle_{\g}+(-\A\nabla^{\g}\Phi+\Delta_{\g}\Phi)\cdot.
\ee
We proceed to determine the phase function $\Phi(t,x)$ such that the eikonal equation 
\bel{eikonal}
\mathcal{S}\Phi=0 \quad \text{on}\quad \M
\ee
is satisfied. The amplitude functions $c_{k,\rho}(t,x)$ for $k=1,2$ are constructed such that the transport equations
\bel{transport}
\mathcal T_{\A_{1,\rho}}c_{1,\rho}=0\quad \text{and}\quad \mathcal T_{-\A_{2,\rho}}c_{2,\rho}=0 \quad \text{on}\quad \M
\ee
hold. Let us start with the eikonal equation. Existence of global smooth solutions to this equation is not guaranteed in general, but owing to the assumption that the manifold is simple, we can find plenty of such solutions. Indeed for the remainder of this section, we will be working in the $z$ coordinate system defined earlier. Recall that this coordinate system is well-defined on $\hat{\M}$ and the null geodesic $\beta$ on $\hat{\M}$ is represented by $(s,0)$ with $s \in [a_0,b_0]$. Recalling the form of the metric from \eqref{metric}, we pick
\bel{phase}
\Phi(z)=z^1.
\ee
To determine the amplitude functions, we first use \eqref{metric} again to rewrite the transport equations \eqref{transport} as 
\bel{c1r}\pd_{z_0}c_{1,\rho}+\left( \frac{\pd_{z_0}\log |g|}{4}-\frac{(\A_{1,\rho})_0}{2}\right)c_{1,\rho}=0,\ee
\bel{c2r}\pd_{z_0}c_{2,\rho}+\left( \frac{\pd_{z_0}\log |g|}{4}+\frac{(\A_{2,\rho})_0}{2}\right)c_{2,\rho}=0,\ee
where $(\A_{k,\rho})_0:=\A_{k,\rho}\nabla^{\g}\Phi$ for $k=1,2$ and in particular we have
\bel{A tangent} 
(\A_{k,\rho})_0\,|_{\beta}=\A_{k,\rho}\dot{\beta}.
\ee
We can take $c_{k,\rho}$ as follows:
\bel{cc1} 
c_{1,\rho}(z):=|g(z)|^{-1/4}\chi(\frac{|z'|}{\delta})\exp\left(\frac{1}{2}\int_{a_0}^{z^0}[(\A_{1,\rho})_0(s,z')]\,ds\right),
\ee
and
\bel{cc2} 
c_{2,\rho}(z):=|g(z)|^{-1/4}\chi(\frac{|z'|}{\delta})\exp\left(-\frac{1}{2}\int_{a_0}^{z^0}[(\A_{2,\rho})_0(s,z')]\,ds\right),
\ee
where $\chi$ is as defined in Section~\ref{smoothapp} and $\delta<\delta'$ (see \eqref{tubular}). It is clear that the amplitude functions $c_{k,\rho}$ are compactly supported in the set $\mathcal V_\beta$ and as a result
\bel{c2} c_{k,\rho}(s,x)=\partial_tc_{k,\rho}(s,x)=0,\quad \text{for}\quad k=1,2,\ s\in\{0,T\},\ x\in M.\ee

\noindent With the construction of the phase and amplitude functions completed as above, we let 
$$F_{1,\rho}=-L_{\A_1,q_1}\left[c_{1,\rho}e^{ i\rho\Phi}\right],\quad F_{2,\rho}=-L^*_{\A_2,q_2}\left[c_{2,\rho}e^{-i\rho\Phi}\right]$$
and we recall that \eqref{eikonal}-\eqref{transport} imply that 
\bel{f1} F_{1,\rho}= -e^{ i\rho\Phi}\left[L_{\A_{1,\rho},q_1}c_{1,\rho}+i\rho(\A_1-\A_{1,\rho})\nabla^{\g}\Phi\, c_{1,\rho}\right],\ee
\bel{f2} F_{2,\rho}= -e^{ -i\rho\Phi}\left[L^*_{\A_{2,\rho},q_2}c_{2,\rho}+i\rho(\A_2-\A_{2,\rho})\nabla^{\g}\Phi\, c_{2,\rho}\right].\ee
We define the expression $R_{j,\rho}$, $j=1,2$, by the solution of the following IBVP
\bel{eqgo1}\left\{ \begin{array}{rcll} L_{\A_1,q_1}R_{1,\rho} & =  F_{1,\rho},&  (t,x) \in \M ,\\
R_{1,\rho}(0,x)=0,\ \ \partial_t R_{1,\rho}(0,x)&=0,& x\in M\\
 R_{1,\rho}(t,x)=0,& \  & (t,x) \in (0,T)\times \pd M ,& \end{array}\right.
\ee
\bel{eqgo2}\left\{ \begin{array}{rcll} L^*_{\A_2,q_2}R_{2,\rho} & =  F_{2,\rho},&  (t,x) \in \M ,\\
R_{2,\rho}(T,x)=0,\ \ \partial_t R_{2,\rho}(T,x)&=0,& x\in M\\
 R_{2,\rho}(t,x)=0,& \  & (t,x) \in (0,T)\times \pd M .& \end{array}\right.
\ee

In order to complete the construction of the solutions $u_1,u_2$ of \eqref{Gsol}, we only need to check the decay of the expression $R_{k,\rho}$, $k=1,2$, given by \eqref{GO2}. According to \eqref{smooth app}-\eqref{mollified}, we have
\bel{cc}\begin{aligned} &\norm{c_{j,\rho}}_{W^{k,\infty}(\M)}\leq C_k\rho^{\frac{k}{4}},\quad k\in\mathbb N^*,\\
&\norm{c_{j,\rho}}_{W^{1,1}(0,T; L^{2}(M))}\leq C,\end{aligned}\ee
with $C$ and $C_k$ independent of $\rho$. Combining this with \eqref{f1}-\eqref{f2}, we find
$$\norm{F_{j,\rho}}_{L^{p_1}(0,T;L^2(M))}\leq C(\rho^{\frac{1}{2}}+\rho\norm{\A_{j,\rho}-\A_j}_{L^{p_1}(0,T;L^2(M))}),\ j=1,2.$$
Using \eqref{mollified} again and the estimate \eqref{p1a} it follows that
$$\lim_{\rho\to+\infty}\rho^{-1}\norm{R_{j,\rho}}_{H^1(\M)}\leq C\lim_{\rho\to+\infty}\rho^{-1}\norm{F_{j,\rho}}_{L^{p_1}(0,T;L^2(M))}=0,\quad j=1,2.$$
Therefore, in order to prove \eqref{GO2}, it only remains to prove that
\bel{GObis}\lim_{\rho\to+\infty}\norm{R_{j,\rho}}_{\CI(0,T;L^2(M))}=0,\quad j=1,2.\ee

\begin{proof}[Proof of Estimate~\ref{GObis}]
The result for $R_{1,\rho}$ and $R_{2,\rho}$ being similar, we will only consider this claim for $R_{1,\rho}$. In view of Lemma~\ref{integral estimate}, the proof of the estimate will be completed if we show that
\bel{l1g}\lim_{\rho\to+\infty}\norm{F_{*,\rho}}_{L^{p_1}(0,T;L^2(M))}=0,\ee
where $F_{*,\rho}(t,x):=-\int_0^tF_{1,\rho}(s,x)\,ds$. Recall that
\bel{l1n}\begin{aligned}
F_{*,\rho}(t,x)=&\int_0^t\left[e^{ i\rho\Phi(\tau,x)}\left[L_{\A_{1,\rho},q_1}c_{1,\rho}(\tau,x)+i\rho (\A_1-\A_{1,\rho})\nabla^{\g}\Phi\, c_{1,\rho}(\tau,x)\right]\right]\,d\tau\\
=&\underbrace{\int_0^t\left[e^{ i\rho\Phi(\tau,x)}\left[(L_{\A_{1,\rho},q_1}-q_1)c_{1,\rho}(\tau,x)+i\rho(\A_1-\A_{1,\rho})\nabla^{\g}\Phi\, c_{1,\rho}(\tau,x)\right]\right]\,d\tau}_{I}\\
&+\underbrace{\int_0^te^{ i\rho\Phi(\tau,x)}q_1c_{1,\rho}(\tau,x)\,d\tau}_{II}.\end{aligned}\ee
To analyze the terms $I$ and $II$ we will integrate by parts in the $\tau$ variable and note that by equation \eqref{phase} we have
\bel{nonvanishing}
\pd_\tau \Phi(\tau,x)=\frac{1}{\sqrt{2}}\neq 0.
\ee
For the term $I$, using the fact that $\A\in W^{1,1}(0,T;L^2(M))$ and equation \eqref{c2}, we can integrate by parts, with respect to $\tau \in(0,t)$, and write
\bel{l1h}\begin{aligned}\frac{\sqrt{2}}{2}\,I&= -i\rho^{-1}e^{i\rho\Phi(t,x)}\left[(L_{\A_1,q_1}-q_1)c_{1,\rho}(t,x)+i\rho\left( (\A_1-\A_{1,\rho})\nabla^{\g}\Phi)\right) c_{1,\rho}(t,x)\right]\\
&\ \ \ \ +i\rho^{-1}\int_0^t\left[e^{ i\rho\Phi(\tau,x)}\left[\partial_t\A_1\nabla^{\g}c_{1,\rho}(\tau,x)+(L_{\A_1,q_1}-q_1)\pd_tc_{1,\rho}(\tau,x)\right]\right]\,d\tau\\
&\ \ \ \ -\int_0^t\left[e^{ i\rho\Phi(\tau,x)}\left[\left( (\A_1-\A_{1,\rho})\nabla^{\g}\Phi)\right) \partial_t c_{1,\rho}(\tau,x)\right]\right]\,d\tau\\
&\ \ \ \ -\int_0^t\left[e^{ i\rho\Phi(\tau,x)}\left[\left( (\pd_t\A_1-\pd_t\A_{1,\rho})\nabla^{\g}\Phi )\right) c_{1,\rho}(\tau,x)\right]\right]\,d\tau\\
&=S_1+S_2+S_3+S_4.\end{aligned}\ee
For the term $S_1$, we can apply \eqref{cc} and \eqref{mollified} to write
$$\begin{aligned}\norm{S_1}_{L^{p_1}(0,T;L^2(M))}&\leq C\rho^{-1}\norm{c_{1,\rho}}_{W^{2,\infty}(\M)}+\left(\norm{\A_1-\A_{1,\rho}}_{L^{p_1}(0,T;L^2(M))}\right)\norm{c_{1,\rho}}_{L^{\infty}(\M)}\\
&\leq C(\rho^{-\frac{1}{2}}+\norm{\A_1-\A_{1,\rho}}_{L^{p_1}(0,T;L^2(M))})=o(1).\end{aligned}$$
For the term $S_2$ we similarly write
\bel{l1j}\begin{aligned}&\norm{\rho^{-1}\int_0^t\left[e^{ i\rho\Phi(\tau,x)}\left[\partial_t\A_1\nabla^{\g}c_{1,\rho}(\tau,x)+(L_{\A_1,q_1}-q_1)\pd_tc_{1,\rho}(\tau,x)\right]\right]\,d\tau}_{L^{p_1}(0,T;L^2(M))}\\
&\leq C\rho^{-1}(\norm{\A_1}_{W^{1,1}(0,T;L^2(M))})\norm{c_{1,\rho}}_{W^{2,\infty}(\M))}\leq C\rho^{-\frac{1}{2}},\end{aligned}\ee
with $C$ independent of $\rho$. For the terms $S_3$ and $S_4$, let us first assume that $\A_1\in \mathcal C^2([0,T];L^2( M))$. Then, integrating by parts with respect to $\tau\in(0,t)$ and applying \eqref{cc}, \eqref{mollified}, we have
$$\norm{\int_0^t\left[e^{ i\rho\Phi(\tau,x)}\left[\left( (\A_1-\A_{1,\rho})\nabla^{\g}\Phi \right) \partial_t c_{1,\rho}(\tau,\cdot)\right]\right]\,d\tau}_{L^{p_1}(0,T;L^2(M))}\leq C\rho^{-\frac{1}{2}},$$
$$\norm{\int_0^t\left[e^{ i\rho\Phi(\tau,x)}\left[\left( (\pd_t\A_1-\pd_t\A_{1,\rho})\nabla^{\g}\Phi \right) c_{1,\rho}(\tau,\cdot)\right]\right]\,d\tau}_{L^{p_1}(0,T;L^2(M))}\leq C\rho^{-\frac{1}{2}},$$
with $C$ independent of $\rho$. Then, applying  the density of $\mathcal C^2([0,T];L^2( M))$ in $W^{1,1}(0,T;L^2(M)$, we deduce that 
$$\lim_{\rho\to \infty} \|S_3\|_{L^{p_1}(0,T;L^2(M))}=\lim_{\rho\to \infty}\|S_4\|_{L^{p_1}(0,T;L^2(M))}=0.$$
Combining the above estimates we conclude that
$$\lim_{\rho\to+\infty}\norm{I}_{L^{p_1}(0,T;L^2(M))}=0.$$
Moreover, in a similar way to the terms $S_3$ and $S_4$, using a density argument combined with \eqref{cc}, we have
$$\lim_{\rho\to+\infty}\norm{II}_{L^{p_1}(0,T;L^2(M))}=0.$$
This completes the proof of estimate~\ref{GObis}.
\end{proof}

\section{Reduction to the light ray transform and the proof of uniqueness}
\label{uniqueness}

\subsection{Reduction to the light ray transform of $\A_1-\A_2$ and proof of Theorem~\ref{t1}}

Suppose $\A_j,q_j$ for $j=1,2$ satisfy regularity conditions \eqref{regularity} and consider their extensions to $\hat{\M}$ and smooth approximations $\A_{j,\rho}$ satisfying \eqref{mollified}. We assume that $\Lambda_{\A_1,q_1}=\Lambda_{\A_2,q_2}$ and proceed to show that for every $\beta \subset \R \times M$ the following identity holds:
\bel{lightray A}
\mathcal L_{\beta} \A =0, 
\ee
where $\A:=\A_1-\A_2$. We start by considering a maximal null geodesic $\beta \subset \mathcal D$ and extend it to $\hat{\M}$. Define $u_j$ in energy space \eqref{energyspace} to be solutions of \eqref{Gsol} taking the form \eqref{go1}-\eqref{go2} with the properties described in the previous section. Let 
$$f_1:=u_1|_{(0,T)\times \pd M}\in H^1_0((0,T]\times \pd M)\quad \text{and}\quad f_2:=u_2|_{(0,T)\times \pd M} \in H^1_0([0,T)\times \pd M).$$ 
Applying Lemma~\ref{allesandrini} we deduce that:
\bel{asymp uniq}
0=\langle (\Lambda_{\A_1,q_1}-\Lambda_{\A_2,q_2})f_1,f_2\rangle=\int_\M \left[\frac{u_2 \A\nabla^{\g}u_1-u_1 \A\nabla^{\g}u_2}{2}+(q-\frac{1}{2}\div_{\g}\A)u_1u_2\right]\,dV_{\g},
\ee
where $q:=q_1-q_2$. Using the Sobolev embedding \eqref{soem}, and the bounds \eqref{cc}-\eqref{GO2}, we write
\bel{corbound}
\begin{aligned}
|\rho^{-1}\int_{\M} R_{j,\rho}\A \nabla^{\g}R_{k,\rho}\,dV_{\g}|&\lesssim \rho^{-1}\|\A\|_{L^{\infty}(\M)}\|R_{k,\rho}\|_{L^2(\M)}\|R_{j,\rho}\|_{H^1(\M)}=o(1),\\
|\rho^{-1}\int_{\M}e^{\pm i\rho\Phi}Q R_{k,\rho}c_{j,\rho}\,dV_{\g}|&\lesssim \rho^{-1}\|Q\|_{L^{p_1}(0,T;L^{p_2}(M))}\|R_{k,\rho}\|_{\CI(0,T;L^2(M))}=o(\rho^{-1}),\\
|\rho^{-1}\int_{\M}QR_{j,\rho}R_{k,\rho}\,dV_{\g}|&\lesssim\rho^{-1}\|Q\|_{L^{p_1}(0,T;L^{p_2}(M))}\|R_{j,\rho}\|_{\CI(0,T;L^2(M))}\|R_{k,\rho}\|_{\CI(0,T;H^1(M))}=o(1).
\end{aligned}
\ee
for $j,k=1,2$ and $Q=q-\frac{1}{2}\div_{\g}\A$. Dividing equation \eqref{asymp uniq} by $\rho$, using \eqref{go1}-\eqref{go2} and applying the latter bounds, we observe that
\[
\lim_{\rho \to \infty} \int_{\M} \A\nabla^{\g}\Phi\, c_{1,\rho}c_{2,\rho}\,dV_{\g} =0.
\]
Recall from \eqref{cc1}-\eqref{cc2} that $c_{k,\rho}$ are compactly supported on $\mathcal V_\beta$. Recalling that $\A=0$ outside of $(0,T)\times M$ (both $\A_1$, $\A_2$ vanish there), and additionally using \eqref{smoothapp}, we write
\[
\lim_{\rho \to \infty} \int_{\mathcal V_\beta} \A_{\rho}\nabla^{\g}\Phi\, c_{1,\rho}c_{2,\rho}\,dV_{\g} =0.
\]
which reduces to
\[
\lim_{\rho \to \infty} \int_{(a_0,b_0)\times B(0,\delta)} (\A_{\rho})_0(z^0,z')\,\chi(\frac{|z'|}{\delta})^2\exp\left(\frac{1}{2}\int_{a_0}^{z^0}(\A_{\rho})_0(s,z')\,ds\right)\,dz^0\,dz'=0,
\]
where $(\A_{\rho})_0=\A\nabla^{\g}\Phi$. Observing that 
$$(\A_{\rho})_0(z^0,z')\exp\left(\frac{1}{2}\int_{a_0}^{z^0}(\A_{\rho})_0(s,z')\,ds\right)=\frac{d}{dz^0}\exp\left(\frac{1}{2}\int_{a_0}^{z^0}(\A_{\rho})_0(s,z')\,ds\right),$$
together with \eqref{cont} and vanishing of $\A$ in the exterior of $\M$, we simplify the former equation to obtain
\[
\int_{B(0,\delta)} \chi(\frac{|z'|}{\delta})^2\exp\left(\frac{1}{2}\int_{a_0}^{b_0}(\A)_0(s,z')\,ds\right)\,dz'=0,
\]
Finally, by taking $\delta \to 0$, and observing that $(\A)_0(s,0)=\A\dot{\beta}$, we observe that
\[
\mathcal L_{\beta} \A\in 4\pi i\mathbb Z.
\]
Note that the above claim holds for any null geodesic $\beta \subset \mathcal D$. Recall from the hypothesis of Theorem~\ref{t1} that $\A$ is supported on the set $\mathcal E$. Thus, we can conclude that the latter equality holds for any null geodesic in $\R \times M$. Let $\beta=(s_0+t,\gamma(t))$ for some $s_0$ and consider a one-parameter family of null geodesics $\beta_s=(s_0+s+t,\gamma(t))$. Since $\A$ is continuous and since $\mathcal L_{\beta_s}=0$ for $s$ large, we conclude that equality \eqref{lightray A} holds. 
Applying statement (ii) in Proposition~\ref{lightray} completes the proof of Theorem~\ref{t1}.

\subsection{Reduction to the light ray transform of $q_1-q_2$ and proof of Theorem~\ref{t2}}

We will assume throughout this section that the additional regularity assumptions \eqref{additional} hold. Applying Theorem~\ref{t1} implies that there exists $\psi \in \CI_0^1(\M)$ such that $\A_1=\A_2+\bar{d}\psi$. Clearly,
$$ \Delta_{\g} \psi =\div_{\g} (\A_1-\A_2) \in L^{p_1}(0,T;L^{\infty}(M)).$$
Let us now define $\tilde{\A}_2=\A_2+\bar{d}\psi$ and $\tilde{q}_2=q_2+\frac{1}{2}\Delta_{\g}\psi-\frac{1}{2}\A_2\nabla^{\g}\psi-\frac{1}{4}\langle\nabla^{\g}\psi,\nabla^{\g}\psi\rangle_{\g}$. Lemma~\ref{DNgauge} applies to show that 
\bel{dneq2}\Lambda_{\A_1,q_1}=\Lambda_{\tilde{\A}_2,\tilde{q}_2}=\Lambda_{\A_1,\tilde{q}_2}.\ee
Analogously to the previous section, we start by considering a null geodesic $\beta \subset \mathcal D$ and extend it to $\hat{\M}$. Define $u_j$ in energy space \eqref{energyspace} to be solutions of \eqref{Gsol} corresponding to differential operators $L_{\A_1,q_1}$ and $L^*_{\A_1,\tilde{q}_2}$, taking the form \eqref{go1}-\eqref{go2} and with the properties described in Section~\ref{GOsection}. Let 
$$f_1:=u_1|_{(0,T)\times \pd M}\in H^1_0((0,T]\times \pd M)\quad \text{and}\quad f_2:=u_2|_{(0,T)\times \pd M} \in H^1_0([0,T)\times \pd M).$$ 
Applying Lemma~\ref{allesandrini} again, we deduce that:
\bel{asymp uniq2}
0=\langle (\Lambda_{\A_1,q_1}-\Lambda_{\A_1,\tilde{q}_2})f_1,f_2\rangle=\int_\M qc_{1,\rho}c_{2,\rho}\,dV_{\g},
\ee
where $q:=q_1-\tilde{q}_2 \in L^{p_1}(0,T;L^{\infty}(M))$. Recall that $c_{1,\rho}, c_{2,\rho}$ are supported in the tubular set $\mathcal V_\beta$ near the null geodesic $\beta$. Estimate \eqref{GObis} implies that
\[
\begin{aligned}
|\int_{\M}qc_{k,\rho}R_{j,\rho}\,dV_{\g}|&\leq\|q\|_{L^{p_1}(0,T;L^2(M))}\|c_{k,\rho}\|_{L^{\infty}(\M)}\|R_{j,\rho}\|_{\CI(0,T;L^2(M))}=o(1),\\
|\int_{\M}qR_{1,\rho}R_{2,\rho}\,dV_{g}|&\leq \|q\|_{L^{p_1}(0,T;L^{\infty}(M))}\|R_1\|_{\CI(0,T;L^2(M))}\|R_2\|_{\CI(0,T;L^2(M))}=o(1).
\end{aligned}
\]
We now use the $z$ coordinate system and note that by taking the limit as $\rho \to \infty$ and using equations \eqref{cc1}-\eqref{cc2} with the preceding correction term bounds, we have
$$\int_{(a_0,b_0)\times B(0,\delta)} q(z^0,z') \chi(\frac{|z'|}{\delta})^2 \,dz^0\,dz'=0.$$
The arguments in Section~\ref{FIO} apply to deduce that 
$$ \mathcal L_{\beta} \,q =0.$$
Together with statement (i) in Proposition~\ref{lightray}, we conclude that equation \eqref{t2c} holds.

\section{Inversion of the light ray transform}
\label{lightray section}
This section is concerned with the proof of Proposition~\ref{lightray}. Recalling Section~\ref{main}, we will identify maximal null geodesics $\beta \subset \R \times M$ with triplets $(s,x,v) \in \R \times \pd_-SM$. Let us first recall the unique inversion of light ray transform on smooth functions. This is reproduced here as some of the arguments are necessary for the extension of the proof to $L^1(0,T;L^2(M))$ functions.
\subsubsection{Inversion of light ray transform for smooth functions}
For $f \in C_c^\infty(\R \times M)$, the transform $\mathcal L f(s,x,v)$ is compactly supported in $s$. Inversion of $\mathcal L$ is based on the following Fourier slicing in time
\begin{align*}
&\widehat{\mathcal L f}(\tau, x,v) = \int_\R e^{-i\tau s} \mathcal Lf(s,x,v)\, ds
= \int_0^{\tau_+(x,v)} \int_\R e^{-i\tau s} f(r+s, \gamma(r;x,v)) \,ds\, dr
\\&
= \int_0^{\tau_+(x,v)} e^{i\tau r} \int_\R e^{-i\tau t} f(t, \gamma(r;x,v))\, dt\, dr
= \int_0^{\tau_+(x,v)} e^{i\tau r} \widehat f(\tau, \gamma(r;x,\xi)) \,dr.
\end{align*}
In particular, $
\widehat{\mathcal L f}(0, x,v) = \mathcal I(\widehat f(0,\cdot))(x,v).
$
Straightforward differentiation gives the following lemma.

\begin{lem}\label{lem_Fourier_slice}
For $f \in C_c^\infty(\R \times M)$, $k=0,1,\dots$, and $(x,v) \in \pd_- SM$ it holds that 
    \begin{align}\label{Fourier_slice}
\pd_\tau^k \widehat{\mathcal L f}(\tau, x,v)|_{\tau = 0}
&= \mathcal I(\pd_\tau^k\widehat f(\tau,\cdot)|_{\tau = 0})(x,v) + \sum_{j=0}^{k-1} \binom k j \mathcal R_{k-j}(\pd_\tau^k\widehat f(\tau,\cdot)|_{\tau = 0})(x,v),
    \end{align}
where 
    \begin{align*}
\mathcal R_j f(x,v) &=
\int_0^{\tau_+(x,v)} (ir)^{j} f(\gamma(r,x,v))\,dr,
\quad f \in C_c^\infty(M).
    \end{align*}
\end{lem}

If $\mathcal I$ is injective then $\mathcal L f = 0$ implies that $\pd_\tau^k\widehat f(\tau,\cdot)|_{\tau = 0} = 0$ for all $k=0,1,\dots$.
As $f$ is compactly supported in $t$, the Fourier transform $\widehat f$ is analytic in $\tau$. Hence $f=0$ in this case. 
\subsubsection{A localization property}

We have the following natural localization property.

\begin{lem}\label{lem_loc}
Let $U \subset \R$ and $V \subset \pd_- SM$ be open. 
Define $W$ to be the set of points $(t,x) \in \R \times M$ such that 
$t = r + s$ and $x = \gamma(r;y,v)$ for some $r \in [0, \tau_+(y,v)]$, $s \in U$ and $(y,v) \in V$.
Suppose that $\chi \in C^\infty(\R \times M)$ satisfies $\chi|_W = 1$. Then 
$$
\mathcal L f|_{U \times V} = \mathcal L (\chi f)|_{U \times V}, \quad f \in \E'(\R \times M).
$$
In particular, for any $f \in \E'(\R \times M)$ there are $a,b \in \R$ such that the support of $\mathcal L f$ is contained in $[a,b] \times \pd_-SM$.
\end{lem}
\begin{proof}
The claimed localization clearly holds when $f \in C^\infty_0(\R \times M)$. For a distribution $f \in \E'(\R \times M)$ we choose a sequence of functions $f_j \in C_0^\infty(\R \times M)$ such that $f_j \to f$ in $\E'(\R \times M)$. Then 
$$
\mathcal L f|_{U \times V} = 
\lim_{j\to\infty} \mathcal L f_j|_{U \times V} = 
\lim_{j\to\infty} \mathcal L (\chi f_j)|_{U \times V} = 
\mathcal L (\chi f)|_{U \times V}.
$$

There is $a_0 \in \R$ such that $f = 0$ in $(-\infty, a_0) \times M$. If $s < a_0 - T$ then the non-trapping assumption implies that the light ray $\beta(r) = (r+s, \gamma(r;y,v))$ does not intersect $\supp(f)$ for any $(y,v) \in \pd_-SM$. Now setting $U = (-\infty, a)$, with $a = a_0 - T - 1$, and $V = \pd_-SM$, we can choose $\chi$ so that $\chi = 1$ in $W$ and $\chi = 0$ in $\supp(f)$.
Then $\mathcal L f$ vanishes in $(-\infty, a) \times \pd_-SM$. Similarly we can get an upper bound for the support with respect to time. 
\end{proof}
\subsubsection{On partial Fourier transform in time}

On a product manifold $\R \times M$ we define the partial Fourier transform in time by 
$$
\pair{\widehat f(z), \phi}_{\E' \times C^\infty(M)}
= \pair{f, e^{-izt} \otimes \phi}_{\E' \times C^\infty(\R \times M)}, \quad f \in \E'(\R \times M),\ z \in \C.
$$
It follows from \cite[Th. 2.1.3]{H1} that $z \mapsto \pair{\widehat f(z), \phi}$
is smooth and for all $j=1,2,\dots$,
$$
\pd_z^j \pair{\widehat f(z), \phi} = \pair{f, \pd_z^j e^{-izt} \otimes \phi}, \quad \pd_{\bar z} \pair{\widehat f(z), \phi} = 0.
$$
The latter identity says that $z \mapsto \pair{\widehat f(z), \phi}$ is analytic, and the former implies that 
the map $f \mapsto \pd_z^j \widehat f(z)|_{z=0}$ is continuous from $\E'(\R \times M)$ to $\E'(M)$.

Let $a,b \in \R$ and consider $L^2((a,b) \times M)$
as a subspace of $L^2(\R; E)$ with $E = L^2(M)$.
Then the above definition of $\widehat f(z)$ coincides with
the usual definition of the Fourier transform on $L^2(\R;E)$. Let us recall that the Fourier transform on $L^2(\R;E)$ is a unitary isomorphism as $E$ is a Hilbert space, see e.g. the discussion on p. 16 of \cite{LM}. It is also easy to see that the map $f \mapsto \pd_z^j \widehat f(z)|_{z=0}$ is continuous from $L^2((a,b) \times M)$ to $L^2(M)$.

\subsubsection{Geodesic ray transform on $L^2$ functions}

Since $\pd M$ is strictly convex, $\mathcal I$ extends as a map from $L^2(M)$ to $L^2(\pd_- SM)$ with a suitably chosen measure on $\pd_- SM$ (see for example \cite[Th. 4.2.1]{S}). In what follows, we will therefore assume that $\mathcal I$ is a map from $L^2(M)$ to $L^2(\pd_- SM)$. 

\subsubsection{The remainder operator $\mathcal R_j$ on $L^2$ functions}

Let us consider the operators $\mathcal R_j$, $j=1,2,\dots$, defined in Lemma \ref{lem_Fourier_slice}.
For $f \in C_c^\infty(M)$ it holds that
$$
|\mathcal R_j f(x,v)| \le L^j \int_0^{\tau_+(x,v)} |f(\gamma(r,x,v))| \,dr = L^j\, \mathcal I(|f|)(x,v),
\quad (x,v) \in \pd_- SM,
$$
where $L=\textrm{Diam}(M)$. Therefore 
$$
\norm{\mathcal R_j f}_{L^2(\pd_-SM)} 
\le 
L^j \norm{\mathcal I(|f|)}_{L^2(\pd_-SM)} 
\le
C \norm{f}_{L^2(M)},
$$
and $\mathcal R_j$ has a unique continuous extension as a map from $L^2(M)$ to $L^2(\pd_- SM)$.

\subsubsection{The inversion}

Let $f \in L^1((0,T);L^2(M))$ and choose a sequence of functions $f_j \in C_c^\infty((0,T) \times M)$ such that $f_j \to f$ in $L^1((0,T);L^2(M))$. Then $\mathcal L f_j \to \mathcal L f$ in $\D'(\R \times \mathbb \pd_- SM)$. As $\mathcal L f$ and $\mathcal L f_j$ are compactly supported in time by Lemma \ref{lem_loc}, also $\pd_z^k \widehat{\mathcal L f_j}(0) \to \pd_z^k \widehat{\mathcal L f}(0)$ in $\D'(\pd_- SM)$.
Furthermore, $\pd_z^k \widehat{f_j}(0) \to \pd_z^k \widehat{f}(0)$ in $L^2(M)$.
Finally, using the $L^2$-continuity of $\mathcal I$ and $\mathcal R_k$, we see that the identity (\ref{Fourier_slice}), that holds for each $f_j$, holds also for $f$ by passing to the limit. 

Recalling that $\mathcal I$ is injective on $L^2(M)$ for simple manifolds $(M,g)$ (see for example \cite{AS} or \cite{Sh}), we see that  $\mathcal L f = 0$ implies that $\pd_z^k \widehat{f}(0) = 0$, as a function in $L^2(M)$, for all $k=0,1,\dots$. For any $\phi \in C_c^\infty(M)$ all the derivatives of the analytic function $\pair{\widehat{f}(z), \phi}$ vanish at the origin. Hence $\pair{\widehat{f}(z), \phi}$ vanishes identically. Therefore $\widehat{f}$ vanishes as a function in $L^2(\R;E)$ with $E = L^2(M)$. We conclude that $f = 0$.

\section*{Acknowledgments}
A.F acknowledges the support from the EPSRC grant EP/P01593X/1 and Y.K. acknowledges the support from the Agence Nationale de la Recherche grant ANR-17-CE40-0029. The authors would also like thank Lauri Oksanen for helpful discussions and his contributions to parts of this paper.

\end{document}